\documentclass[11pt]{amsart}

\usepackage{epsfig,amsmath,amsfonts,latexsym}
\usepackage[abs]{overpic}
\usepackage{color}

\newtheorem{theorem}{Theorem}
\newtheorem{proposition}[theorem]{Proposition}
\newtheorem{lemma}[theorem]{Lemma}
\newtheorem{question}[theorem]{Question}
\newtheorem{cor}[theorem]{Corollary}

\theoremstyle{definition}
\newtheorem{definition}[theorem]{Definition}

\theoremstyle{remark}
\newtheorem{remark}[theorem]{Remark}

\def\co{\colon\thinspace}
\def\dfn#1{{\em #1}}
\def\R{\mathbb{R}}

\numberwithin{theorem}{section}
\theoremstyle{plain}

\begin{document}

\title{Contact surgery and symplectic caps}

\author{James Conway}

\address{Department of Mathematics \\ University of California, Berkeley \\ Berkeley \\ California}

\email{conway@berkeley.edu}

\author{John B. Etnyre}

\address{School of Mathematics \\ Georgia Institute
of Technology \\  Atlanta  \\ Georgia}

\email{etnyre@math.gatech.edu}

\address{}

\email{}


\begin{abstract}
In this note we show that a closed oriented contact manifold is obtained from the standard contact sphere of the same dimension by contact surgeries on isotropic and coisotropic spheres. In addition, we observe that all closed oriented contact manifolds admit symplectic caps.  
\end{abstract}

\maketitle

\section{Introduction}
The fact that all contact 3--manifold can be obtained from the standard contact sphere $(S^3,\xi_{std})$ by a sequence of Legendrian $(\pm1)$--surgeries was first proved in \cite{DingGeiges04} and has had a profound impact on the study of such manifolds. One can think of a Legendrian knot in a contact 3--manifold as both an isotropic and a coisotropic sphere, and with this in mind, our first main result is a generalization of this result to higher dimensions. 
\begin{theorem}\label{mainsurgery}
Any oriented contact $(2n-1)$--manifold $(M,\xi)$ can be obtained from $(S^{2n-1},\xi_{std})$ by a sequence of contact surgeries on isotropic and coisotropic spheres. 
\end{theorem}
In this paper a contact structure on an oriented manifold will always be assumed to be co-oriented. 
\begin{remark}
The expert will notice that this result follows, once one has appropriately defined contact surgeries, from any result that a given contact manifold $(M,\xi)$ is Weinstein cobordant to a contact manifold to which $(S^{2n-1},\xi_{std})$ is also Weinstein cobordant. As we point out below, this easily follows from work of  Bowden, Crowley, and Stipsicz \cite{BowdenCrowleyStipsicz15}, just as the $3$--dimensional version follows from \cite{EtnyreHonda02a}. However, illuminating the purely 3--dimensional nature of the construction in \cite{DingGeiges04} has had great influence in contact geometry, and we hope a similar illumination in higher dimensions will prove similarly useful. 

We also note that in \cite{EliashbergMurphyPre15} Eliashberg and Murphy ingeniously construct symplectic cobordisms structure on any smooth cobordism between an overtwisted contact manifold and any contact manifold induced on the upper boundary of the cobordism by an almost-complex structure on the cobordism. While it would seem this result could also be used to establish the above theorem, this is not the case, as their cobordisms are Liouville and are not obtained by simple Weinstein handle attachments. We do note, however, that what we call contact surgeries on isotropic and coisotropic spheres, they discuss as direct and inverse Weinstein surgeries. Our terminology reflects our focus on the contact-geometric nature of the operation we are discussing rather than the symplectic cobordism structure on which Eliashberg and Murphy were focusing. 
\end{remark}

\begin{remark}
Since we are trying to focus on the contact geometric nature of the construction in the main theorem, we spend some time in Section~\ref{coiso} giving a precise description of the spheres on which we can do surgery. The description of, and local models for, isotropic spheres are well-known, but the description of, and local models for, coisotropic spheres is not as well understood. One could simply define their neighborhoods to be the boundaries of co-cores of Weinstein handles, and when focusing on Weinstein cobordisms (as in, say, \cite{EliashbergMurphyPre15}), this is quite natural to do. When focusing on the contact surgery picture, however, it makes more sense to have a specific description of the submanifold on which the surgery is taking place.

In the absence of a standard neighborhood theorem for coisotropic submanifolds of contact manifolds, we show that the coisotropic spheres that we consider actually have a standard neighborhood.  In higher dimensions, it is somewhat complicated to precisely give the appropriate definition of coisotropic spheres and their standard neighborhoods; this is one of the key technical points in the paper and is carried out in Section~\ref{coiso}. From the Weinstein handle point of view, the upshot of our work is that to attach an inverse handle to a sphere, it is sufficient to find a standard coisotropic sphere, and not an embedding of a neighborhood of the co-core.
\end{remark}

\begin{remark}
While completing a draft of this note, Oleg Lazarev sent a draft of his paper \cite{LazarevPre} in which he establishes Theorem~\ref{mainsurgery} for almost-Weinstein fillable contact manifolds $(M^{2n-1}, \xi)$, using isotropic surgery of arbitrary index and coisotropic surgery of index $n$. 
\end{remark}
\begin{remark}
One of the reasons the surgery description of contact 3--manifolds from \cite{DingGeiges04} is so useful is that in \cite{DingGeigesStipsicz04}, Ding, Geiges, and Stipsicz were able to show how the almost-contact structure underlying a contact structure was affected by Legendrian $(\pm1)$--surgeries.  Also, models of surgeries in dimension 3 have enabled results that translate between surgery on transverse knots and surgery on Legendrian knots, \cite{BaldwinEtnyre13, Conway14?}. Both of these successes should be replicable in higher dimensions as well, and will be the subject of a sequel to this paper. 
\end{remark}

Another useful result in 3--dimensional contact geometry is the existence of symplectic caps for a contact manifold. Recall that a symplectic cap for a contact manifold is a compact symplectic manifold with the contact manifold as its concave boundary. The existence of symplectic caps for Stein fillable manifolds was first proved by Lisca and Mati\'c in \cite{LiscaMatic97}, where they used such caps to give the first examples of homology spheres admitting an arbitrarily large number of tight contact structures.  In \cite{EtnyreHonda02a} the second author and Honda showed that any contact 3--manifold has a symplectic cap, and in fact has infinitely many different such caps. 

\begin{theorem}\label{caps}
Any oriented contact $(2n-1)$--manifold $(M,\xi)$ has infinitely many symplectic caps that are not related by a sequence of blowing-up and blowing-down. 
\end{theorem}
\begin{remark}
The existence of symplectic caps for contact structures was originally established for Stein fillable contact structures by Lisca and Mati\'c \cite{LiscaMatic97} and for overtwisted contact structures by Eliashberg and Murphy \cite{EliashbergMurphyPre15}. 
\end{remark}
\begin{remark} \label{finitemaximals}
While completing a draft of this note the authors learned that Oleg Lazarev 
had also obtained Theorem~\ref{caps}, see \cite{LazarevPre}. In addition, he showed that the partial order on contact manifolds coming from Weinstein cobordisms in dimensions above 3 has upper bounds for finite sets. That is, given any finite collection of contact manifolds, there is some Stein fillable contact manifold $(M^{2n-1},\xi)$, for which any contact manifold in the collection has a Weinstein cobordism to $(M^{2n-1},\xi)$. After hearing of his result, we noticed that the proof of Proposition~\ref{cobordtostein} also establishes this result: namely, there is a ``pre-maximal element" in the partial order. That is, there is a Stein fillable contact manifold $(Y^{2n-1},\eta)$  such that every contact manifold is Weinstein cobordant to a contact manifold obtained from $(Y,\eta)$ by attaching a single Weinstein $n$--handle. From this it is easy to conclude that any finite collection of contact manifolds has an upper bound.

We (as well as Lazarev) leave open the question as to whether there is a maximal element to the partial order coming from Stein cobordisms; however, Lazarev shows that if one allows strong symplectic cobordisms, there is in fact a maximal element.
\end{remark}

Of course, an immediate corollary of the existence of symplectic caps is the following embedding result. 

\begin{cor}\label{embed}
Any strong symplectic filling $(X,\omega)$ of a contact manifold can be embedded into a closed symplectic manifold. 
\end{cor}

One could hope for better and we ask the following question.
\begin{question}
Can any symplectic filling (in particular, a weak symplectic filling) of a contact manifold be embedded into a closed symplectic manifold?
\end{question}
We recall that in dimension 3 the answer is yes, as was shown by Eliashberg \cite{Eliashberg04} and the first author \cite{Etnyre04a}. The existence of such embeddings was crucial to various results in low dimensional topology, see for example \cite{OzsvathSzabo04a}. 

{\bf Acknowledgements.} The authors thank Chris Wendl for helpful comments on the first draft of the paper. The first author was partially supported by NSF grant DMS-1344991; the second author was partially supported by NSF grants DMS-1608684 and~1906414.\hfill$\Box$

\section{Background}

\subsection{Symplectic cobordisms}
We say that a contact manifold $(M,\xi)$ is the {\em convex}, respectively {\em concave}, boundary of a symplectic manifold $(X,\omega)$ if $M$ is a boundary component of $X$ and there is a vector field $v$ defined on a neighborhood of $M$ in $X$ so that $v$ points out of, respectively into, $X$ and $\mathcal{L}_v \omega= \omega$ and $\xi=\ker (\iota_v\omega)$.

A {\em symplectic cobordism} from the contact manifold $(M_-,\xi_-)$ to the contact manifold $(M_+,\xi_+)$ is a symplectic manifold $(X,\omega)$ with $\partial X= -M_-\cup M_+$, where $M_-$ is a concave and $M_+$ is a convex boundary component of $X$. Given a symplectic cobordism, we will use $\partial_\pm X$ to denote the contact manifold $(M_\pm,\xi_\pm)$.  If $\partial_-X=\emptyset$ then $(X,\omega)$ is called a {\em symplectic filling} of $\partial_+X=(M_+,\xi_+)$ and if $\partial_+X=\emptyset$ then $(X,\omega)$ is called a {\em symplectic cap} for $\partial_-X=(M_-,\omega_-)$. 

We now recall a well-known gluing result for symplectic cobordisms.
\begin{theorem}\label{glue}
Let $(X_i,\omega_i)$ be a symplectic cobordism for $i=0,1$.  If there is a diffeomorphism $f: \partial_+X_0 \to \partial_-X_1$ that preserves the contact structures, then $X_0$ can be glued to $X_1$ via $f$ to obtain a cobordism $X$ and there is a symplectic structure $\omega$ on $X$ that restricts to $\omega_0$ on $X_0$ and away from a collar neighborhood of $\partial_-X_1$ restricts to a positive multiple of $\omega_1$. In particular, $(X,\omega)$ is a symplectic cobordism from $\partial_-X_0$ to $\partial_+X_1$. 
\end{theorem}

\subsubsection{Stein cobordisms}
If $X$ is a complex manifold, we call a function $\phi:X\to \R$ (strictly) plurisubharmonic if $\omega_\phi=-dJ^*(d\phi)$ is a symplectic form on $X$, where $J$ is the integrable almost-complex structure on $X$.  A special type of symplectic cobordism is a Stein cobordism. If $(X,\omega)$ is a symplectic cobordism and $X$ has a complex structure and a plurisubharmonic function $\phi:X\to \R$ for which $\partial_-X$ and $\partial_+X$ are regular level sets and $\omega=\omega_\phi$, then we say $(X,\omega)$ is a {\em Stein cobordism}. A Stein filling is frequently called a {\em Stein domain}. 

\subsubsection{Weinstein cobordisms}
We can slightly weaken the notion of Stein cobordism to Weinstein cobordism. We call a symplectic cobordism $(X,\omega)$ a {\em Weinstein cobordism} if there is a primitive $\lambda$ for $\omega=d\lambda$ and a Morse function $f:X\to \R$ such that $v_\lambda$ is gradient like for $f$, where $v_\lambda$ is the unique vector field satisfying $\iota_{v_\lambda}\omega=\lambda$. Given a contact manifold $(M,\xi)$, where $\xi$ is defined by the $1$--form $\alpha$, then a trivial example of a Weinstein cobordisms is $([a,b]\times M, d(e^t\alpha))$, with the Morse function being projection to $[a,b]$.  Cieliebak and Eliashberg characterized Weinstein cobordisms as follows.
\begin{theorem}[Cieliebak and Eliashberg 2012, \cite{CieliebakEliashberg12}]\label{CE}
Let $X$ be a cobordism of dimension $2n\geq 6$ admitting a Morse function $f:M\to \R$ with critical points of index at most $n$. If there is a non-degenerate $2$--form $\eta$ on $X$ that agrees with a trivial Weinstein structure near $\partial_-X$, then there is a $1$--form $\lambda$ on $X$, such that $(X,d\lambda,f)$ is a Weinstein cobordisms that agrees with the one near $\partial_-X$, and $d\lambda$ is homotopic through non-degenerate $2$--forms to $\eta$. 
\end{theorem}
\begin{remark}\label{EliashRem}
Eliashberg's characterization of Stein manifolds \cite{Eliashberg90a} implies that if one uses Theorem~\ref{glue} to glue a Stein filling of a contact manifold to a Weinstein cobordism, the resulting symplectic manifold is also a Stein domain. 
\end{remark}

It is well-known (and clear from the previous theorem) that a $2n$--dimensional Weinstein cobordism has a handle decomposition with handles of index at most $n$. A general Weinstein cobordism can be built by attaching Weinstein handles to the standard symplectic $2n$--ball or to a piece of the symplectization of a contact manifold, see \cite{CieliebakEliashberg12, Weinstein91}. A {\em Weinstein handle} of index~$k$ is $h^k=D^k\times D^{2n-k}$ with a symplectic structure so that $\partial_-h^k=(\partial D^k)\times D^{2n-k}$ is concave, and $\partial_+h^k=D^k\times \partial D^{2n-k}$ is convex. Moreover, $D^k\times \{0\}$ is isotropic and its intersection with $\partial_-h^k$ is an isotropic $S^{k-1}$ in the contact structure induced on $\partial_-h^k$. Similarly, $\{0\}\times D^{2n-k}$ is coisotropic and its intersection with $\partial_+h^k$ is a coisotropic $S^{2n-k-1}$ in the contact structure induced on $\partial_+h^k$.

Thus, the attaching sphere of a Weinstein $k$--handle is an isotropic $S^{k-1}$. Recall that to attach a handle one needs a trivialization of the normal bundle, so we consider the normal bundle of an isotropic submanifold of a contact manifold. 

Recall that a submanifold $S$ of a contact manifold $(M,\xi)$ is called {\em isotropic} if $T_xS\subset \xi_x$ for all $x\in S$. If $\alpha$ is a contact form defining $\xi$ then one may easily check that $T_xS$ is an isotropic subspace of $\xi$ with the symplectic structure $d\alpha$. In other words the symplectic orthogonal $TS^{\perp_\omega}$ of $TS$ in $\xi$ contains $TS$. The {\em conformal symplectic normal bundle} of $S$ is defined to be the quotient bundle 
\[
CSN(S)=(TS^{\perp_\omega})/TS.
\]
There is a natural symplectic form induced on each fiber of $CSN(S)$ by $d\alpha$ (as one changes $\alpha$ the symplectic structure changes by a conformal factor). It is easy to see that the tangent bundle of $M$ along $S$ splits as
\[
TM=\xi \oplus TM/\xi = (CSN(S)\oplus TS \oplus (\xi/ TS^{\perp_\omega}))\oplus \underline{\R}.
\]
A standard argument shows that $\xi/TS^{\perp_\omega}$ is isomorphic to $T^*S$ and so the normal bundle to $S$ in $M$ is 
\[
\nu(S)=CSN(S)\oplus T^*S\oplus \underline{\R}.
\]
Restricting to the case of isotropic spheres $S$, we notice that $T^*S\oplus\underline{\R}$ is trivial, and hence a trivialization of the normal bundle to $S$ is given by a (conformally) symplectic trivialization of $CSN(S)$. 

\begin{lemma}\label{isogerm}
The germ of a contact structure along an isotropic sphere is determined up to contactomorphism by the isomorphism class of its conformal symplectic normal bundle. 
\end{lemma}

See Theorem~2.5.8 in \cite{Geiges08}, for example. Thus given an isotropic sphere $S^{k-1}$ in the convex boundary of a symplectic manifold with a choice of trivialization of its conformal symplectic normal bundle, one can attach a Weinstein $k$--handle by identifying a neighborhood of the isotropic sphere with $\partial_-h^k$. See \cite{CieliebakEliashberg12, Weinstein91}.

\begin{remark}\label{handlemodle}
For later use, we now recall a model for a Weinstein $k$--handle. We begin by thinking of $\R^{2n}$ as the product of $T^*\R^k$ with the symplectic structure $\sum_{i=1}^k dq_i\wedge dp_i$ and $\R^{2(n-k)}$ with the symplectic structure $\sum_{i=1}^{n-k} dx_i\wedge dy_i$. Now for any $a, b > 0$, let $D_a$ be the disk of radius $a$ in the $p_i$ subspace and $D_b$ the disk of radius $b$ in the $q_i,x_i,y_i$ subspace. Then set $H_{a,b}= D_a\times D_b$. This is a model for 
the Weinstein $k$--handle $h^k$ with attaching region $\partial_-h^k = (\partial D_a)\times D_b$. We have the expanding vector field 
\[
v=\sum_{i=1}^k -p_i\frac{\partial}{\partial p_i} + 2q_i\frac{\partial}{\partial q_i} + \frac 12 \sum_{i=1}^{n-k}\left( x_i\frac{\partial}{\partial x_i} + y_i\frac{\partial}{\partial y_i}\right)
\]
that can be used to attach the handle as in \cite{Weinstein91}. We notice that if $\omega$ is the symplectic structure on $\R^{2n}$ above, then $\iota_v\omega$ restricted to $\partial_-h^k = (\partial D_a)\times D_b$ and $\partial_+h^k = D_a\times \partial D_b$ is a contact form. 
\end{remark}

\subsubsection{Almost-complex cobordisms}
We now turn to a weaker notion of cobordism. We first recall that an {\em almost-contact structure} on a manifold $M$ is a hyperplane field $\eta$ together with a non-degenerate $2$--form $\omega$ on $\eta$. Recall that the existence of $\omega$ on $\eta$ is equivalent to the existence of an almost-complex structure on $\eta$. 

Now recall that if $J$ is an almost-complex structure on $X$ and $\Sigma$ is a hypersurface then we can consider the set $\xi$ of complex tangencies to $\Sigma$. This will be a hyperplane field on $\Sigma$. If we orient $\Sigma$ by a transverse vector field $v$ along $\Sigma$ such that $Jv$ is tangent to $\Sigma$ then we can choose a $1$--form $\alpha$ such that $\xi=\ker\alpha$ and $\alpha(Jv)>0$. We say that $\Sigma$ is {\em $J$--convex}, respectively {\em $J$--concave}, if $d\alpha(w,Jw)>0$, respectively $d\alpha(w,Jw)<0$, for all $w\not=0$ in $\xi$. Notice that $\xi$ and $J$ give $\Sigma$ an almost-contact structure. If $X$ is a cobordism from $\partial_-X$ to $\partial_+X$ and $\partial_-X$ is $J$--convex with respect to an inward pointing normal and $\partial_+X$ is $J$--convex with respect to an outward pointing vector field then we say that $(X,J)$ is an almost-complex cobordism from the almost-contact manifold $\partial_-X$ to $\partial_+X$. 

\begin{theorem}[Bowden, Crowley, and Stipsicz 2015, \cite{BowdenCrowleyStipsicz15}]\label{BCS}
For any $2n-1\geq 5$, there is a closed almost-contact manifold $(M_{max},\eta_{max},\omega_{max})$ such that for any other closed oriented almost-contact manifold $(M,\eta,\omega)$ there is an almost-complex cobordism from $M$ to $M_{\max}$ with handles of index less than or equal to $n$ and so that the almost-contact structures are induced by the almost-complex structure. 
\end{theorem}
We have the immediate corollary of this result that we will upgrade in Proposition~\ref{cobordtostein}.
\begin{cor}\label{partialcobord}
Given any contact manifold $(M,\xi)$ of dimension at least 5, there is a Weinstein cobordism from $(M,\xi)$ to a contact manifold $(M',\xi')$ such that the almost-contact homotopy class of $\xi'$ contains a Stein fillable contact structure. 
\end{cor}
\begin{proof}
Given $(M,\xi)$, Theorem~\ref{BCS} gives an almost-complex cobordism $(X,J)$ from $(M, \xi)$ to $(M_{max},\eta_{max},\omega_{max})$. Now Theorem~\ref{CE} says that $J$ may be homotoped so that $X$ is a Stein (and hence Weinstein) cobordism from $(M,\xi)$ to some contact manifold $(M_{max},\xi')$. Now notice that given any Stein fillable contact manifold $(\widehat{M},\widehat{\xi})$ we can similarly get a Weinstein cobordism to some contact manifold $(M_{max}, \xi'')$. By construction, see Remark~\ref{EliashRem}, $(M_{max},\xi'')$ is Stein fillable and in the same almost-contact homotopy class of $(M_{max},\xi')$.
\end{proof}

\subsection{Symplectic caps}
Part of Theorem~3.2 in \cite{LiscaMatic97} says that a Stein domain can always be embedded in a closed symplectic manifold. 
\begin{theorem}[Lisca and Mati\'c 1997, \cite{LiscaMatic97}]\label{LM}
If $X$ is a Stein domain with plurisubharmonic function $\phi:X\to \R$, then there is a holomorphic embedding of $X$ into a smooth projective variety such that the K\"ahler form on the variety pulls back to the standard symplectic form $\omega_\phi= -dJ^*(d\phi)$ on $X$, where $J$ is the integrable almost-complex structure on $X$. 
\end{theorem}
An immediate corollary of this is the existence of symplectic caps for Stein fillable contact structures.
\begin{cor}\label{steincap}
Let $(M^{2n-1},\xi)$ be a Stein fillable contact structure. There is a symplectic manifold $(W^{2n},\omega)$ such that $(M,\xi)$ is the concave boundary of $(W, \omega)$. 
\end{cor}
\begin{proof}
Let $X$ be the Stein filling of $(M,\xi)$ and let $(Y,\omega)$ be the smooth projective variety for which there is a holomorphic embedding $f:X\to Y$ as in Theorem~\ref{LM}. Clearly $W=\overline{Y-f(X)}$ with the symplectic form $\omega|_W$ is the desired symplectic cap for $(M,\xi)$.  
\end{proof}

\subsection{Overtwisted contact structures}
Having an almost-contact structure is clearly a necessary condition for a manifold to admit a contact structure. In 2015, Borman, Eliashberg, and Murphy actually showed that it was also sufficient, \cite{BormanEliashbergMurphy15}. In particular, in that paper they introduced the notion of an overtwisted contact structure in all odd dimensions. Their definition is a bit technical so we do not recall it here, but merely state their main theorem. 

\begin{theorem}[Borman, Eliashberg, and Murphy 2015, \cite{BormanEliashbergMurphy15}]\label{ot}
Given any almost-contact structure $(\eta,\omega)$ on $M$, there is a homotopy through almost-contact structure from $(\eta, \omega)$ to an overtwisted contact structure. Moreover, any two overtwisted contact structures that are homotopic through almost-contact structures are isotopic through contact structures. 
\end{theorem}

Another property of overtwisted contact structures we will need involves loose Legendrian knots as defined by Murphy in \cite{Murphy??}. Again, the precise definition is not important to us here, just the fact that given any Legendrian $\Lambda$ in a contact manifold $(M,\xi)$ of dimension at least 5, once $\Lambda$ has been stabilized near a cusp, it is loose. The main result we need is the following. 

\begin{theorem}[Casals, Murphy, and Presas 2015, \cite{CasalsMurphyPresas}]\label{+1surgot}
Legendrian $(+1)$--surgery on a loose Legendrian sphere in a contact manifold produces an overtwisted contact manifold.
\end{theorem}

\begin{remark}
See Remark~\ref{contactpm1} for a definition of Legendrian $(\pm1)$--surgery, where we interpret such surgeries as isotropic and coisotropic surgeries on a Legendrian sphere. 
\end{remark}

\subsection{Convex surfaces}
Recall that a hypersurface $\Sigma$ in a contact manifold $(M^{2n-1},\xi)$ is called {\em convex} if there is a vector field $v$ transverse to $\Sigma$ whose flow preserves $\xi$, {\em i.e.\@} $v$ is a {\em contact vector field}. The {\em dividing set} of $\Sigma$ is
\[
\Gamma_\Sigma= \{x\in \Sigma: v(x)\in \xi_x\}.
\]
This is a co-dimension 1 submanifold of $\Sigma$. We also have the {\em positive and negative regions} of $\Sigma$ which are the following components of $\Sigma\setminus \Gamma_\Sigma$,
\[
\Sigma_\pm = \{x\in \Sigma: v(x) \text{ is $\pm$ transverse to } \xi_x\}.
\]
Here are a couple of simple properties of this set-up.
\begin{enumerate}
\item The dividing set $\Gamma_\Sigma$ is a contact submanifold of $(M,\xi)$.
\item The closures $(\pm 1)^n\overline \Sigma_\pm$ of the positive and negative regions are ideal Liouville domains in $(M,\xi)$ that fill the given contact structure on $\Gamma_\Sigma$. 
\end{enumerate}
To check these properties we first recall that in \cite{Giroux17} (see also \cite{MassotNiederkrugerWendl13}), Giroux defined an {\em ideal Liouville domain} to be a compact manifold with boundary $W$ together with a symplectic form $\omega$ on the interior of $W$ so that $\omega$ has a primitive $\lambda$ for which there is a function $u:W\to \R$ such that $u\lambda$ extends to a $1$--form on $W$ that induces a contact form on $\partial W$. We say an ideal Liouville domain $(W,\omega)$ fills a contact structure $\xi_\partial$ on $\partial W$ if the contact from induced on $\partial W$ by a primitive for $\omega$ as in the definition above, defines $\xi_\partial$. Finally, we say an ideal Liouville domain $(W,\omega)$  embeds in a contact manifold $(M,\xi)$ if $W$ is a submanifold of $M$, so that a contact form for $\xi$ induces the ideal Liouville structure. More specifically, there is a contact form $\alpha$ defining $\xi$ for which $\alpha|_{\partial W}$ is a contact form on $\partial W$ and some rescaling of $\alpha$ (on the interior of $W$) is a primitive for $\omega$ on the interior of $W$.  

Now to verify the properties mentioned above, let $\alpha$ be a contact $1$--form defining $\xi$ on $M$ and set $u=\alpha(v)$. Clearly $\Gamma_\Sigma=\{u^{-1}(0)\}$ (and one may check that 0 is a regular value of $u$) and $\Sigma_\pm =\{\pm u>0\}$. It is also not hard to see that $\alpha$ restricts to $\Gamma_\Sigma$ to give a contact form, see \cite[Section~I.3.B]{Giroux91}. One may check that where $u\not=0$, $u\alpha$ is contact form and $v$ is its Reeb vector field. Thus $d(\alpha/u)$ restricted to $\Sigma_\pm$ gives $(\pm 1)^n\overline \Sigma_\pm$ the structure of an ideal Liouville domain in $(M,\xi)$ that fills the contact structure induced by $\alpha$ on $\Gamma_\Sigma$. 

We now define an {\em abstract convex surface} to be a closed $2n$--dimensional manifold $S$ together with 
\begin{enumerate}
\item a separating hypersurface $\Gamma$, a contact structure $\xi_\Gamma$ on $\Gamma$, and 
\item symplectic forms $\omega_\pm$ on $\omega_\pm$ on $(\pm1)^{n}\overline S_\pm$, were $S-\Gamma=S_+\cup S_-$, that give $(\pm1)^{n}\overline S_\pm$ the structure of ideal Liouville domains
\end{enumerate}
so that there is a global $1$--form $\beta$ on $S$ and function $u:S\to \R$ that has $\Gamma=u^{-1}(0)$ as a regular level set, $\beta$ is a contact form on $\Gamma$ inducing $\xi_\Gamma$, and $d(\beta/u)=\omega_\pm$ on $S_\pm$.  

It is now a simple computation that given an abstract convex surface $(S,\Gamma,\omega_\pm)$ with the $1$--form $\beta$ and function $u$ above, that 
\[
\alpha= u\, dt +\beta
\]
is a contact form on $S\times \R$ in which $\partial/\partial t$ is a contact vector field, $S\times \{t\}$ is a convex surface with $\Gamma_{S\times \{t\}}=\Gamma$ and $\alpha$ inducing $\beta$ on each $S\times\{t\}$. 

Given an abstract convex surface  $(S,\Gamma,\omega_\pm)$, the set of $\beta$ and $u$ as in the definition form a convex set. Thus, a simple application of Gray's stability result implies that the contact structure on $S\times \R$ depends only on $(S,\Gamma,\omega_\pm)$. More specifically, given two choices for $(\beta, u)$ and $(\beta',u')$ as above, there will be a $1$--parameter family of such equations, and Moser's method produces a $t$--invariant vector field whose time--$1$ flow takes $\ker(u\, dt+ \beta)$ to $\ker(u'\, dt + \beta')$ and $S\times \{0\}$ will be taken by this diffeomorphism to the graph of a function $S\to \R$. 

We also recall, {\em cf.\@} \cite[Lemma 2.3.7]{HondaHuangPre}, that there is a contact embedding from $S\times (-b,b)$ to $S\times (-a,a)$ for any positive $a$ and $b$, where the contact structure on both manifolds is given by $u\, dt + \beta$. 

Combining the above observations, we have the following result that will be central to our construction. 
\begin{theorem}\label{convexgluing}
Let $(M_i,\xi_i)$, $i=0,1$, be two contact manifolds and $f$ an orientation reversing diffeomorphism from a boundary component $\Sigma_0$ of $M_0$ to a boundary component $\Sigma_1$ of $M_1$. Suppose $\Sigma_i$ is a convex surface in $(M_i,\xi_i)$. If $f$ is a contactomorphism $\Gamma_{\Sigma_0}$ to $\Gamma_{\Sigma_1}$ and is a symplectomorphism of the ideal Liouville domains $(\Sigma_i)_{\pm}$, then one may remove a collar neighborhood of $\Sigma_1$ from $M_1$ and glue it to $M_0$ so the resulting manifold has a contact structure into which each of the $(M_i,\xi_i)$ contact embed.  \hfill$\Box$
\end{theorem}

\section{Contact surgery}
\subsection{Isotropic Spheres}
We begin by describing a standard model for an isotropic $(k-1)$--sphere in a contact $(2n-1)$--dimensional manifold. While this is a well-known result, we present it here for the sake of notation needed later. Consider the contact manifold 
\[
\left(J^1(S^{k-1})\oplus\underline{\R}^{2(n-k)}, dz-\lambda + \sum_{i=1}^{n-k} x_i\, dy_i - y_i\, dx_i\right),
\]
where the jet space $J^1(S^{k-1})=T^*S^{k-1}\oplus\underline{\R}$, $\lambda$ is the Liouville form on $T^*S^{k-1}$, $z$ is the coordinate on $\R$, and the $x_i$, $y_i$ are coordinates on $\R^{2(n-k)}$. It is easy to check that the zero section $Z$ in this bundle is an isotropic $(k-1)$--sphere with trivial conformal symplectic normal bundle. By Lemma~\ref{isogerm}, the germ of a contact structure is determined by this bundle, so if $S$ is any isotropic $(k-1)$--sphere in a contact manifold $(M^{2n-1},\xi)$ with a trivialization of the conformal symplectic normal bundle, then any neighborhood of $S$ contains a tubular neighborhood contactomorphic to a tubular neighborhood of $Z$ in the above model.

Using this model we may easily establish the following regular neighborhood result.

\begin{lemma}\label{stdnbhdiso}
If $S$ is an isotropic $(k-1)$--sphere in a contact manifold $(M^{2n-1},\xi)$ with trivial conformal symplectic normal bundle $CSN(S)$, then in any open set containing $S$ there is a standard neighborhood $N$ with convex boundary $\partial N = S^{k-1}\times S^{2n-k-1}$ contactomorphic to an $\epsilon$--neighborhood $N_\epsilon$ of the zero section $Z$ in $J^1(S^{k-1})\oplus\underline{\R}^{2(n-k)}$. Moreover, under this contactomorphism, the dividing set on the boundary is 
\[
\Gamma_{\partial N} = (\partial N)\cap \{z=0\}\cong S^{k-1}\times S^{2n-k-2}
\]
and the positive and negative regions of $\partial N$ are 
\[
(\partial N)_\pm =  (\partial N_\epsilon)\cap\{\pm z > 0\}
\]
with ideal Liouville forms given by $d(\alpha/z)$, where $\alpha$ is the standard contact from on $J^1(S^{k-1})\oplus\underline{\R}^{2(n-k)}$ discussed above.
\end{lemma} 
\begin{proof}
Using Lemma~\ref{isogerm} we only need to verify the result for neighborhoods of the zero section in the model presented above. To this end denote by $N_\epsilon$ the $\epsilon$--disk bundle in $J^1(S^{k-1})\oplus\underline{\R}^{2(n-k)}$. For any positive $\epsilon$ this will be our standard neighborhood of the zero section. Now notice that the vector field 
\[
v= z\frac{\partial}{\partial z} + \sum_{i=1}^{k-1} p_i\frac{\partial}{\partial{p_i}} + \frac 12 \sum_{i=1}^{n-k} \left(x_i\, \frac{\partial}{\partial{x_i}}+y_i\, \frac{\partial}{\partial{y_i}}\right)
\]
is a contact vector field that points transversely out of all the neighborhoods $N_\epsilon$ and its flow identifies all such neighborhoods, so there is indeed a single model for a standard neighborhood. The vector field $v$ also shows that $S_\epsilon=\partial N_\epsilon\cong S^{k-1}\times S^{2n-k-1}$ is a convex hypersurface. 

Let $f = \iota_v\alpha$ where $\alpha$ 
is the contact form in our model above. Clearly $f(z,q,p, x,y)= z$. Thus $\Gamma_{\partial N_\epsilon}$ is as claimed in the lemma. 
\end{proof}

\begin{remark}
Notice that the attaching region for a Weinstein $k$--handle, see Remark~\ref{handlemodle}, is a neighborhood of an isotropic $S^{k-1}$ and in that model a (scaled) radial normal vector field is the contact vector field. 
\end{remark}

\subsection{Coisotropic spheres}\label{coiso}
In this section we define the standard coisotropic sphere and show it has a standard neighborhood. Before we given the formal definition, it is helpful to keep in mind that we are modeling the belt spheres (that is, the boundary of the co-core) of Weinstein $k$--handles. The formal definition is constructed to characterize this sphere.

Recall that a submanifold $S$ of a contact manifold $(M,\xi)$ is called {\em coisotropic} if $T_xS\cap \xi_x$ is a coisotropic subspace of $\xi_x$ for all $x\in S$. That is, the $d\alpha$--orthogonal of $S_\xi:=T_xS\cap \xi_x$ is contained in $S_\xi$, where $\alpha$ is a $1$--form defining $\xi$. 

Recall that the sphere  $S^{2n-k-1}$ is homeomorphic to the join of spheres $S^{2(n-k)-1}*S^{k-1}$, that is, as the quotient of $S^{2(n-k)-1}\times  S^{k-1} \times [0,1]$, where each of the sets $S^{2(n-k)-1}\times \{p\}\times \{0\}$ and $\{1\}\times \{q\}\times S^{k-1}$ is collapsed to a point. Let $K_1=S^{k-1}$ and $K_2=S^{2(n-k)-1}$ be the locus of collapsed points in the quotient space. Note that $S^{2n-k-1}-K_1= S^{2(n-k)-1}\times D^k$ and $S^{2n-k-1}-K_2= D^{2(n-k)}\times S^{k-1}$.
 
\begin{definition}
For $k \leq n$, an embedded $S^{2n-k-1}$ with trivial normal bundle in a contact $(2n-1)$--manifold $(M,\xi)$ will be called a \dfn{standard coisotropic sphere} if 
\begin{enumerate}
\item $\xi$ is transverse to $S^{2n-k-1}$ except along $K_1=S^{k-1}$ where it is tangent to $S^{2n-k-1}$,
\item $R\times S^{k-1}$ is isotropic for all  rays $R\subset D^{2(n-k)}$ from the origin,  
\item $S^{2(n-k)-1}\times \{p\}$ is transverse to $\xi$ and $\xi$ induces the standard contact structure on $S^{2(n-k)-1}\times \{p\}$, for each $p\in D^k$.
\end{enumerate}
\end{definition}
We first note that the standard coisotropic sphere is indeed a coisotropic submanifold of $(M,\xi)$. To see this, notice that on $S^{2n-k-1}-K_1=S^{2(n-k)-1}\times D^k$, the distribution $S^{2n-k-1}_\xi= (TS^{2n-k-1})\cap \xi$ is $\xi'\oplus TD^k$, where $\xi'$ is the standard contact structure on $S^{2(n-k)-1}$. So if $\alpha$ is a contact form for $\xi$, then $d\alpha$ restricted to $S^{2n-k-1}_\xi$ will vanish precisely on $TD^k$, and since $S^{2n-k-1}_\xi$ is codimension $k$ in $\xi$ we see that $TD^k$ is the $d\alpha$--orthogonal of $S^{2n-k-1}_\xi$ in $\xi$. Similarly, along $K_1=S^{k-1}$ we have that $S^{2n-k-1}_\xi= TS^{2n-k-1}= TD^{2(n-k)}\oplus TS^{k-1}$, and it is easy to see (see the model constructed in the proof of Lemma~\ref{stdnbhdcoiso}) that $TD^{2(n-k)}$ is a symplectic subspace of the symplectic form $d\alpha$ on $\xi$. Moreover, $TS^{k-1}$ is isotropic, so the $d\alpha$--orthogonal to $S^{2n-k-1}_\xi$ is $TS^{k-1}$. Thus our standard coisotropic sphere really is coisotropic. 

\begin{lemma}\label{coisomodel}
The standard coisotropic sphere determines the germ of a contact structure near the sphere. 
\end{lemma}

\begin{remark}
In the proof of this result we will also see that a map between two standard coisotropic spheres that preserves the induced characteristic distribution, determines a map between the normal bundles of the sphere. Thus there are no extra choices, like the framing of the conformal symplectic normal bundle for isotropic spheres, when determining a neighborhood of a coisotropic sphere. 
\end{remark}
\begin{proof}[Proof of Lemma~\ref{coisomodel}]
Let $\xi$ and $\xi'$ be two contact structures defined on $S^{2n-k-1}\times D^k$ so that $S^{2n-k-1}\times\{0\}$ is a standard coisotropic sphere in both contact structures. Let $\alpha$ and $\alpha'$ be contact forms for $\xi$ and $\xi'$, respectively. Also, denote by $\beta$ the standard contact form on $S^{2(n-k)-1}$.

On the tangent space to $S^{2n-k-1}-(K_1\cup K_2)=S^{2(n-k)-1}\times S^{k-1}\times (0,1)$, the conditions defining the standard coisotropic sphere imply that $\alpha = g(p,r)\beta$ and $\alpha'=h(p,r)\beta$ for non-negative functions $g,h\co S^{k-1}\times (0,1)\to \R$.  To extend over $K_2$, we must have that as $r$ goes to $1$ both functions approach a positive function. Similarly, to extend over $K_1$ and be a contact form, as $r$ goes to $0$ we must have that the functions are asymptotic to $r^2$ (times any function in $p$ that is positive; see Example~2 below). Notice on $\{p\}\times D^k$ both $1$--forms vanish, according to the definition of a standard coisotropic sphere. Thus, there is a non-zero function $f:S^{2n-k-1}\to \R$ such that $\alpha=f\alpha'$ on the tangent space to $S^{2n-k-1}\times \{0\}$, and so by rescaling $\alpha'$, we can assume that $\alpha$ and $\alpha'$ agree on the tangent space to $S^{2n-k-1}\times \{0\}$.

We now proceed in two steps. In the first step we arrange for the contact structures to agree in a neighborhood of $K_1=S^{k-1}\subset S^{2n-k-1}\times \{0\}$ where $\xi$ and $\xi'$ are both tangent to $S^{2n-k-1}\times \{0\}$. Then the final step to to isotope to make the contact structures the same in a neighborhood of $S^{2n-k-1}\times \{0\}$. 

For the first step we choose a complex structure $J$ and $J'$ compatible with $d\alpha$ and $d\alpha'$ on $\xi$ and $\xi'$. A neighborhood of $K_1$ in $S^{2n-k-1} \times \{0\}$ is $D^{2(n-k)} \times K_1$ (here we are taking the second disk factor very small). The normal bundle to $S^{2n-k-1}\times \{0\}$ near $K_1$ is spanned by $J(T(\{p\}\times S^{k-1}))$ and $X_\alpha$ and also by $J'(T(\{p\}\times S^{k-1}))$ and $X_{\alpha'}$, where $X_\alpha$ is the Reeb vector field for $\alpha$, and similarly for $X_{\alpha'}$. There is an isotopy of a neighborhood of $S^{2n-k-1}\times \{0\}$ in $S^{2n-k-1}\times D^k$ that is fixed on $S^{2n-k-1}\times \{0\}$ that takes one parameterization of the normal bundle to the other in a neighborhood of $K_1$.
We already know that $\alpha=\alpha'$ and $d\alpha=d\alpha'$ on $D^{2(n-k)} \times K_1$, but now after the isotopy, we claim that they are equal on the entire tangent space at $K_1$. Indeed, the normal bundle to $S$ restricted to $K_1$ in $\xi'$, parameterized by $J'$, is taken by the isotopy to the normal bundle in $\xi$, parameterized by $J$, so we see that $\alpha$ and $\alpha'$ now have the same kernel on the tangent space to $S^{2n-k-1}\times D^k$ along $K_1$, and $d\alpha=d\alpha'$ on the kernel. Moreover, since our isotopy takes $X_\alpha$ to $X_{\alpha'}$ along $K_1$, we see that $\alpha=\alpha'$ and $d\alpha=d\alpha'$ on the entire tangent space to $S^{2n-k-1}\times D^k$ along $K_1$, as claimed.

Now consider $\alpha_t=t\alpha + (1-t)\alpha'$. This is constant on $K_1$, and hence a contact form near $K_1$ for all $t$. Moreover, $\frac{\partial \alpha_t}{\partial t}=0$ on the tangent space to $D^{2(n-k)} \times K_1$. One may now proceed as in the proof of Theorem~2.5.23 of \cite{Geiges08} to construct an isotopy fixed on $D^{2(n-k)} \times K_1$ that will take $\xi'$ to $\xi$ in a neighborhood of $K_1$, and so in particular, we can choose contact $1$--forms for them that also agree here. Note that the vector field used to define the isotopy will vanish on $D^{2(n-k)} \times K_1$, and hence can be extended over a neighborhood of $S^{2n-k-1}\times \{0\}$ so that it is 0 on $S^{2n-k-1}\times \{0\}$. Thus we can define an isotopy of a neighborhood of $S^{2n-k-1}\times \{0\}$ that arranges for $\xi$ and $\xi'$ to agree in a neighborhood of $K_1$ and does not change the contact planes on $S^{2n-k-1}\times \{0\}$. We finish this step by rescaling our contact forms $\alpha$ and $\alpha'$ if necessary so that they are again equal on $S^{2n-k-1} \times \{0\}$.

Proceeding now to the second step, we can assume that $\xi=\xi'$ in a neighborhood $U$ of $K_1$ in $S^{2n-k-1}\times D^k$ and that $J=J'$ on $U$ too. 
Notice that $[S^{2n-k-1}-K_1]\times \{0\}= S^{2(n-k)-1}\times D^k$ and all the $D^k$'s are isotropic. The normal bundle to $[S^{2n-k-1}-K_1]\times \{0\}$ can be taken to be $J(T(\{p\}\times D^k))$ or $J'(T(\{p\}\times D^k))$. Building an isotopy outside of $U$ taking one normal bundle to the other we see that $\alpha$ and $d\alpha$ agree with $\alpha'$ and $d\alpha'$ on the tangent space to $S^{2n-k-1}\times D^k$ along $S^{2n-k-1}\times \{0\}$, and from the first step agree on $U$ as well. A standard application of Moser's method gives the desired contactomorphism.  Note that the isotopy is fixed where $\alpha$ and $\alpha'$ already agree, so in particular is fixed in a neighborhood of $K_1$.
\end{proof}

As we prove the following neighborhood theorem for a standard coisotropic sphere, we will come up for a model neighborhood, just like we have for the isotropic spheres.
\begin{lemma}\label{stdnbhdcoiso}
If $S$ is a standard coisotropic $(2n-k-1)$--sphere in a contact manifold $(M^{2n-1},\xi)$, then in any open set containing $S$ there is a standard neighborhood $N$ with convex boundary $\partial N= S^{k-1}\times S^{2n-k-1}$ contactomorphic to a model neighborhood built below. Moreover, under this contactomorphism, $\partial N$ can be identified with the convex surface in the boundary of a model for a neighborhood of an isotropic sphere in Lemma~\ref{stdnbhdiso}.  
\end{lemma} 

\begin{proof}
By Lemma~\ref{coisomodel} we merely need to build a standard model with the desired properties. To this end, we now build a model contact structure on a neighborhood of the standard coisotropic sphere. The description of $S^{2n-k-1}$ as the join of two spheres at the beginning of this subsection shows that $S^{2n-k-1}=A\cup_\phi B$ where 
\[
A=D^{2(n-k)}\times S^{k-1}, \text{ and } B= S^{2(n-k)-1}\times D^{k},
\]
with $D^{2(n-k)}$ and $D^k$ being  unit disks,
and
\[
\phi\co (A-(\{0\}\times S^{k-1}))\to (B-(S^{2(n-k)-1}\times \{0\}))\co (r,x,y)\mapsto (x,1-r,y)
\]
where 
$
A-(\{0\}\times S^{k-1})= [0,1]\times S^{2(n-k)-1}\times S^{k-1}
$
and
$
B-(S^{2(n-k)-1}\times \{0\})= S^{2(n-k)-1}\times [0,1]\times S^{k-1}.
$

We will build a contact structure on a neighborhood $S^{2n-k-1}\times \R^k_n$ of our sphere by building a contact structure on each piece $A\times \R_n^k$ and $B\times \R_n^k$ and seeing that they fit together to give a global contact structure. Here, $\R^k_n$ should be seen as the normal direction to the sphere when we use this as a model for a coisotropic sphere inside a manifold. To this end, we begin by considering two basic examples. 

\noindent
{\bf Example 1.} Consider $\R^{2k}$ as the cotangent bundle of $\R^k$. If $(q_1,\ldots, q_k)$ are coordinates on $\R^k$, then we denote the coordinates on $\R^{2k}$ by $(q_1,\ldots, q_k, p_1, \ldots, p_k)$, where the canonical $1$--form is given by 
\[
\lambda= \sum p_i\, dq_i.
\]
Let $U^*\R^k$ be the unit sphere bundle in $T^*\R^k$, and denote the restriction of the $1$--form $\lambda$ to $U^*\R^k$ by $\beta$. Using the flow of $\sum p_i\partial_{p_i}$, we can identify the complement of the zero section in $T^*\R^k$ with $\R\times  U^*\R^k$, and $\lambda= e^r \beta$. By exponentiating the first coordinate, we can further identify this with $(0,\infty)\times  U^*\R^k$; now $\lambda=r \beta$, where $r$ is the radial coordinate in the cotangent bundle. 

\noindent
{\bf Example 2.} Consider the unit sphere $S^{2(n-k)-1}$ in $\R^{2(n-k)}$. We have the $1$--form $\eta=\sum x_i\, dy_i -y_i\, dx_i$ on $\R^{2(n-k)}$, and its restriction to $S^{2(n-k)-1}$ is denoted $\alpha$. Notice that as above, $\eta= r^2 \alpha$ (the extra factor of $r$ appears since both the $x_i$ and $y_i$ are scaled). 

Clearly $\alpha$ and $\beta$ are contact $1$--forms on $S^{2(n-k)-1}$ and $U^*\R^k$, respectively. One can easily check that 
\[
f(r) \alpha + g(r) \beta
\]
is a contact form on $\R\times S^{2(n-k)-1}\times U^*\R^k$ if $f'g-g'f>0$ and $f$ and $g$ are non-zero. 

We can identify $(0,1)\times S^{2(n-k)-1}\times U^*\R^k$ with a subset of $A\times \R^k_n= D^{2(n-k)}\times S^{k-1}\times \R^k_n$ by the map $(r,x,p,q)\mapsto (rx,p,q)$, where we are thinking of $S^{k-1}\times \R^k_n$ as $U^*\R^k$. Then if $f(r)=r^2$ and $g(r)=1$ near $0$ we see that $f(r)\alpha + g(r)\beta$ extends to all of $A\times D^k_n$ as $\eta+\beta$. 

Similarly, we identify $(0,1)\times S^{2(n-k)-1}\times U^*\R^k$ with a subset of $B\times \R^k_n=  S^{2(n-k)-1}\times D^{k}\times \R^k_n$ by the map $(r,x,p,q)\mapsto (x,(1-r)p,q)$. Then, if $f(r)=1$ and $g(r)=1-r$ near $r=1$, we see that $f(r)\alpha + g(r)\beta$ extends to all of $B\times \R^k_n$ as $\alpha+\lambda$. 

This gives our model contact structure on $S^{2n-k-1}\times \R^k_n$. 

Consider the model for a Weinstein $k$--handle from Remark~\ref{handlemodle}. There the belt sphere $\{0\}\times \partial D_b$ is a standard coisotropic sphere, as can easily be seen by parameterizing it as the join of the $S^{k-1}$ in the $q_i$ space and the $S^{2(n-k)-1}$ in the $x_i, y_i$ space. Thus by Lemma~\ref{coisomodel}, there is a contactomorphism from an arbitrarily small neighborhood of $S^{2n-k-1}$ in our standard model to a neighborhood of $\{0\}\times \partial D_b$ in the Weinstein $k$--handle. So to complete the lemma, we only need to check the last statement for a neighborhood of the belt sphere $\{0\}\times \partial D_b$ in the Weinstein $k$--handle. To this end, notice that if we consider the attaching region $A=(\partial D_a)\times D_b$ (all notation is from Remark~\ref{handlemodle}) and set $A'=(A-({\partial D_a}\times \{0\}))$, then the flow of the vector field $v$ embeds the symplectization $\R\times A'$ of the contact structure induced on $A'$ into $\R^{2n}$. Moreover, if $B=D_a\times \partial D_b$ is the neighborhood of the belt sphere and  $B'=B-(\{0\}\times \partial D_b)$, then $B'$ is transverse to the $\R$ factor of $\R\times A'$. It is a standard fact that the induced contact structure on $B'$ agrees with the one on $A'$. Thus the convex boundary of a neighborhood of the attaching sphere $(\partial D_a)\times \{0\}$ is identified with a convex surface in $B$ that bounds a neighborhood of the belt sphere. 
\end{proof}

\subsection{Contact surgeries}\label{contactsurg}
With all the preliminaries of the previous two subsections out of the way, we can now define contact surgeries. 
\begin{definition}
If $S^{k-1}$ is an isotropic sphere in a contact manifold $(M^{2n-1},\xi)$, then after choosing a trivialization of its conformal symplectic normal bundle we have a standard neighborhood $N$ of $S^{k-1}$ given by Lemma~\ref{stdnbhdiso}. There is also a neighborhood $N'$ of a model coisotropic sphere $S^{2n-k-1}$ with a trivialized normal bundle given in Lemma~\ref{stdnbhdcoiso}. 
Both neighborhoods have convex boundary and there is a diffeomorphism from $\partial N$ to $\partial N'$ that preserves the dividing sets and the complementary ideal Liouville domains. So by Theorem~\ref{convexgluing}, we can remove $N$ from $M$ and glue $N'$ in its place. The result will be called {\em isotropic surgery} on $S^{k-1}$. 

We can similarly start with a standard coisotropic sphere $S^{2n-k-1}$ with a trivialized normal bundle in $(M^{2n-1},\xi)$ and argue as above to replace a neighborhood of it with a neighborhood of an isotropic sphere $S^{k-1}$. This will be called {\em coisotropic surgery} on $S^{2n-k-1}$.  Note that in both isotropic and coisotropic surgery, the setup mandates starting with a {\em framed} sphere.  Different choices of framings may give rise to non-isotopic or non-contactomorphic contact structures, not to mention non-diffeomorphic manifolds.

We will use the general term {\em contact surgery} for either an isotropic surgery or coisotropic surgery. \hfill$\Box$
\end{definition}

\begin{remark}
Given a formal isotropic embedding of $S^k$ into a $(2n-1)$--manifold, for $1 \leq k \leq n-1$, an $h$--principle guarantees the existence of an isotropic sphere in that formal isotopy class (see, for example, Chapter~7 in \cite{CieliebakEliashberg12}).  However, for coisotropic spheres that are not Legendrian, no such $h$--principle exists.  Indeed, for every $n \geq 3$ and $1 \leq k \leq n-2$, there exist formal coisotropic embeddings of $S^{2n-k-1}$ into a contact $(2n-1)$--manifold that are not formally isotopic to a standard coisotropic sphere (see Theorem~1.3 and Remark~6.1 in \cite{GhigginiNiederkrugerWendl2016}).
\end{remark}

We have the obvious lemma.
\begin{lemma}
After an isotropic surgery, there is a canonical coisotropic sphere on which coisotropic surgery will undo the isotropic surgery, and vice versa.\hfill$\Box$
\end{lemma}
It is easy to see how this relates to Weinstein handle attachment.
\begin{proposition}\label{surgeryisweinstein}
Given a contact manifold $(M,\xi=\ker \alpha)$, let $(W,\omega)$ be the result of attaching a Weinstein $k$--handle to an isotropic sphere $S^k$ in $\{b\}\times M$ in $([a,b]\times M, d(e^t \alpha))$.  Denote by $(M',\xi')$ the convex boundary of $W$. Let $S^{2n-k-1}$ be the coisotropic sphere in $M'$ that is the boundary of the co-core of the Weinstein handle.

The contact manifold $(M',\xi')$ is obtained from $(M,\xi)$ by isotropic surgery on $S^{k-1}$ and $(M,\xi)$ is obtained from $(M',\xi')$ by coisotropic surgery on $S^{2n-k-1}$. 
\end{proposition}
\begin{proof}
This is clear from the proof of Lemma~\ref{stdnbhdcoiso}. 
\end{proof}
An immediate corollary is the following.
\begin{cor}\label{weinsteinimplysurgey}
If $(M,\xi)$ is Weinstein cobordant to $(M',\xi')$, then each of $(M,\xi)$ and $(M',\xi')$ can be obtained from the other by contact surgeries.  \hfill$\Box$
\end{cor}

\begin{remark}\label{contactpm1}
Notice that a Legendrian sphere is both isotropic and coisotropic.
In the literature it is common to refer to an isotropic surgery on a Legendrian sphere as {\em Legendrian surgery} or {\em Legendrian $(-1)$--surgery} and coisotropic surgery on a Legendrian sphere as {\em Legendrian $(+1)$--surgery}.

In \cite{Avdek}, Avdek gave a different definition of contact $(\pm 1)$--surgery on Legendrian spheres, based around a generalised Seidel--Dehn twist (see Section 9 there for details); this same definition also appears in \cite{CasalsMurphyPresas}, where Casals, Murphy, and Presas show that the result of contact $(+1)$--surgery on a loose Legendrian knot is overtwisted.  Avdek shows in \cite[Theorem~9.15]{Avdek} that his definition corresponds to Weinstein handle attachment, and hence agrees with our definition, by Proposition~\ref{surgeryisweinstein}.
\end{remark}

\section{Weinstein Cobordisms}\label{scobord}
Our main technical tool to prove Theorems~\ref{mainsurgery} and~\ref{caps} is the following result.
\begin{lemma}\label{preprelim}
Suppose $(M^{2n-1}_i,\xi_i)$ is a contact manifold and $(W_i, \omega_i)$ are Weinstein cobordisms with $\partial_-W_i=(M^{2n-1}_i,\xi_i)$, for $i=1,2$. If $M=\partial_+W_1=\partial_+W_2$ and the contact structure $\xi'_1$ on $\partial_+W_1$ is homotopic through almost-contact structures to the contact structure $\xi_2'$ on $\partial_+W_2$, then there is a fixed contact manifold $(M',\xi')$ and a Weinstein cobordism $(W'_i,\omega_i')$ from  $(M^{2n-1}_i,\xi_i)$ to $(M',\xi')$.
\end{lemma}
\begin{proof}
After an isotopy we can assume that $\xi'_1$ and $\xi'_2$ agree in the neighborhood $U$ of a point. In this neighborhood we can choose a Legendrian sphere $\Lambda$ that has been stabilized, and hence is loose in both $\xi'_1$ and $\xi'_2$. Let $\Lambda'$ be a copy of $\Lambda$ pushed off by the Reeb field and consider $\Lambda$ as a Legendrian sphere in $(M,\xi'_1)$ and $\Lambda'$ as a Legendrian sphere in $(M,\xi'_2)$. Let $(\widehat{M},\widehat{\xi}_1)$ be the result of Legendrian $(+1)$--surgery on $\Lambda$ in $(M,\xi_1')$ and similarly let $(\widehat{M},\widehat{\xi}_2)$ be the result of Legendrian $(+1)$--surgery on $\Lambda'$ in $(M,\xi'_2)$. Since $\Lambda$ and $\Lambda'$  are loose, both of these contact manifolds are overtwisted, by Theorem~\ref{+1surgot}. Moreover, they are in the same almost-contact class (note they are homotopic as plane fields in the complement of $U$ and agree on $U$). Thus, $\widehat{\xi}_1$ is isotopic to $\widehat{\xi}_2$ by Theorem~\ref{ot}, and we assume that they have been isotoped to be the same. 

Let $\widehat{\Lambda}$ and $\widehat{\Lambda'}$ be the dual Legendrian spheres in the surgered manifold $(\widehat{M}, \widehat{\xi}_1)$. We can assume they are disjoint. So Legendrian $(-1)$--surgery on $\widehat{\Lambda}$ will result in $(M,\xi_1')$ and inside $(M,\xi_1')$ we have the image of $\widehat{\Lambda'}$ (which we denote by the same symbol). Further Legendrian $(-1)$--surgery on $\widehat{\Lambda'}$ in $(M,\xi_1')$ will result in some contact manifold $({M'},{\xi'})$, and clearly $(M,\xi_1)$ is Weinstein cobordant to $({M'},{\xi'})$. Reversing the order of the surgeries will show that $(M,\xi_2')$ is also Weinstein cobordant to $({M'},{\xi'})$. Concatenating these cobordisms with the original cobordisms gives the desired result. 
\end{proof}

\begin{proposition}\label{cobordtostein}
Any closed oriented contact manifold $(M,\xi)$ is Weinstein cobordant to a Stein fillable contact structure.
\end{proposition}
\begin{proof}
If $(M,\xi)$ is $3$--dimensional then the result follows from Theorem~1.1 in \cite{EtnyreHonda02a}, so we now assume that the dimension is larger than 3. In this case, Corollary~\ref{partialcobord} gives a Weinstein cobordism from $(M,\xi)$ to a contact manifold $(M',\xi')$ that is in the same almost-contact homotopy class as a Stein fillable contact structure. Now Lemma~\ref{preprelim} gives a Weinstein cobordism from $(M',\xi')$ to a Stein fillable contact structure $(M'',\xi'')$. These cobordisms can be glued to produce a Weinstein cobordism from $(M,\xi)$ to the Stein fillable contact structure $(M'',\xi'')$. 
\end{proof}

\subsection{Symplectic caps}
We can now prove Theorem~\ref{caps} concerning the existence of symplectic caps for any contact manifold. 
\begin{proof}[Proof of Theorem~\ref{caps}]
Let $(M,\xi)$ be any closed oriented contact manifold. Proposition~\ref{cobordtostein} gives a Weinstein cobordism from $(M,\xi)$ to a Stein fillable contact structure $(M',\xi')$. Assume that $X$ is the Stein filling of $(M',\xi')$. The corollary to Lisca and Mati\'c's theorem, Corollary~\ref{steincap}, gives a symplectic cap $(W,\omega)$ for $(M',\xi')$. Using the gluing result Theorem~\ref{glue}, the manifold $(W,\omega)$ and the Stein cobordism from $(M,\xi)$ to $(M',\xi')$ may be glued together to produce the symplectic cap for $(M,\xi)$.  

To get infinitely many caps, we begin by choosing a Stein filling $(W, \omega)$ of an exotic contact structure $\xi$ on $S^{2n-1}$ that is homotopic to the standard contact structure. We take $W$ to be a connected sum of Brieskorn varieties when $n$ is odd, and a connected sum of Brieskorn varieties and products of spheres when $n$ is odd (see, for example, \cite[Lemma~2.17]{BowdenCrowleyStipsicz14}); these can be chosen to give the standard almost contact structure on the boundary. Let $(W_k,\omega_k)$ be the boundary connected sum of $k$ copies of $(W, \omega)$, and let $\xi_k$ be the induced contact structure on $S^{2n-1} = \partial W_k$.  Note that the almost-contact class of $(S^{2n-1}, \xi_k)$ is independent of $k$.

As mentioned in Remark~\ref{finitemaximals}, given any finite $n$, there is a single Stein fillable contact manifold $(M'_n, \xi'_n)$ such that $(M,\xi)\#(S^{2n-1},\xi_k)$ is Weinstein cobordant to $(M'_n, \xi'_n)$ for all $k = 1, \ldots, n$.  We claim that the underlying smooth cobordism from $M=M\#S^{2n-1}$ to $M'_n$ can be chosen to be independent of $n$. Then the union of $W_i$, a $1$--handle joining $[0, 1] \times M$ to $W_i$, the cobordism to $M'_n$, and a fixed cap for $M'_n$ provides $n$ distinct symplectic caps for $(M, \xi)$.  A simple examination of the homology of $W_i$ shows that these caps cannot be obtained from each other by a sequence of blow-ups and blow-downs. Since $n$ was arbitrary, this finishes the proof.

To prove the claim, we choose an almost-complex cobordism from $(M, \xi)$ to $(M_{max}, \eta_{max}, \omega_{\max})$ with handles of index at most $n$, provided by Theorem~\ref{BCS}. For each $k = 1, \ldots, n$, we get a Weinstein cobordism from $(M, \xi)\#(S^{2n-1}, \xi_k)$ to some contact structure $\xi_{max, k}$ on $M_{max}$, by Theorem~\ref{CE}, all in the same almost-contact class. Similarly to the proof of Lemma~\ref{preprelim}, for each $k$, there are $(n-1),$  $n$--handles that we can attach to $(M_{max}, \xi_{max, k})$ that will result in the contact manifold $(M'_n, \xi'_n)$, independent of $k$. As is clear from the proof of Lemma~\ref{preprelim}, the smooth isotopy classes of the attaching spheres for the $n$--handles is independent of $k$. Hence, the smooth cobordism from $M$ to $M'_n$ that is the union of the cobordism to $M_{max}$ and the $n$--handles is independent of $k$, proving the claim. 
\end{proof}

It is now a simple matter to establish Corollary~\ref{embed} that any strong filling of a contact manifold embeds in a closed symplectic manifold. 
\begin{proof}[Proof of Corollary~\ref{embed}]
Given a symplectic filling $(X,\omega)$ of a contact manifold $(M,\xi)$, Theorem~\ref{caps} gives a symplectic cap $(W,\omega_W)$ for $(M,\xi)$. Hence Theorem~\ref{glue} produces a closed symplectic manifold into which $(X,\omega)$ embeds. 
\end{proof}

\subsection{Constructing manifolds via contact surgeries}
With the above in hand, we now establish Theorem~\ref{mainsurgery} concerning the construction of all contact manifolds by contact surgeries from $(S^{2n-1},\xi_{std})$. 
\begin{proof}[Proof of Theorem~\ref{mainsurgery}]
Given any contact manifold $(M,\xi)$, Proposition~\ref{cobordtostein} gives a Weinstein cobordism from $(M,\xi)$ to a Stein fillable contact structure $(M',\xi')$. The Stein cobordism gives a Weinstein cobordism from $(S^{2n-1},\xi_{std})$ to $(M',\xi')$ by simply removing the $0$--handle of the filling. Now Corollary~\ref{weinsteinimplysurgey} implies that $(M,\xi)$ is related to $(M',\xi')$ by contact surgeries, and $(M',\xi')$ is related to $(S^{2n-1},\xi_{std})$ by contact surgeries, hence so are $(M,\xi)$ and $(S^{2n-1},\xi_{std})$.
\end{proof}

\bibliography{references}
\bibliographystyle{plain}

\end{document}